\documentclass[11pt,a4paper]{amsart}
\usepackage{graphicx,multirow,array,amsmath,amssymb}

\newtheorem{theorem}{Theorem}
\newtheorem{proposition}[theorem]{Proposition}

\renewcommand{\arraystretch}{1.2}

\begin{document}

\title{Enumeration on graph mosaics}

\author[K. Hong]{Kyungpyo Hong}
\address{National Institute for Mathematical Sciences, Daejeon 34047, Korea}
\email{kphong@nims.re.kr}
\author[S. Oh]{Seungsang Oh}
\address{Department of Mathematics, Korea University, Seoul 02841, Korea}
\email{seungsang@korea.ac.kr}

\thanks{Mathematics Subject Classification 2010: 05C30, 57M25, 81P99}
\thanks{The corresponding author(Seungsang Oh) was supported by the National Research Foundation of Korea(NRF) grant funded by the Korea government(MSIP) (No. NRF-2014R1A2A1A11050999).}

\begin{abstract}
Since the Jones polynomial was discovered, 
the connection between knot theory and quantum physics has been of great interest.
Lomonaco and Kauffman introduced the knot mosaic system to give a definition of the quantum knot system
that is intended to represent an actual physical quantum system.
Recently the authors developed an algorithm producing the exact enumeration of knot mosaics, 
which uses a recursion formula of state matrices.
As a sequel to this research program,
we similarly define the (embedded) graph mosaic system by using sixteen graph mosaic tiles,
representing graph diagrams with vertices of valence 3 and 4.
And we extend the algorithm to produce the exact number of all graph mosaics.
The magnified state matrix that is an extension of the state matrix is mainly used.
\end{abstract}

\maketitle

\section{Introduction and terminology}

In 1983, Vaughan Jones discovered the famous Jones polynomial
that was an essentially new way of studying knots.
It turned out that the explanation of the Jones polynomial has to do with quantum theory.
The interested reader is referred to other articles \cite{J1, J2, K1, K2, L, LK2, SJ}.
Lomonaco and Kauffman introduced a knot mosaic system to set the foundation for a quantum knot system
in the series of papers \cite{LK1, LK3, LK4, LK5}.
Their definition of quantum knots 
whose basic building block is a mosaic system,
was based on the planar projections of knots and the Reidemeister moves.
They model the topological information in a knot by a state vector in a Hilbert space
that is directly constructed from knot mosaics.

The authors, Lee and Lee have announced
the progress of the research program on finding the total number of knot mosaics
in the series of papers \cite{HLLO1, HLLO2, LHLO, Oh1, OHLL}.
We have developed a partition matrix argument to count small knot mosaics,
and also generalized this argument to give an algorithm for counting all knot mosaics
that uses recurrence relations of so-called state matrices.
This algorithm is further generalized 
for the enumeration of finite two-dimensional regular lattice objects
such as independent vertex sets \cite{OhV1},
which is an outstanding unsolved combinatorial problem in statistical mechanics.

The purpose of this paper is to present a combinatorial extension of 
the definition of knot mosaics to graph theory,
and an application of the algorithm for the enumeration of embedded graph mosaics.
Throughout this paper, a {\em graph mosaic\/} means a mosaic regular projection of 
embedded graphs (possibly multiple components) with vertices of valence 3 and 4 in 3-space
$\mathbb{R}^3$.
The requirement on the number of valence is necessary to superpose in mosaic shapes.

Let $\mathbb{T}_G$ denote the set of the following sixteen symbols that are 
called {\em graph mosaic tiles\/} as in Figure~\ref{fig1}.
This definition is an extended version of the definition of mosaic tiles for knots \cite{LK3}, 
which are the first eleven tiles $T_0$ through $T_{10}$.
The other five tiles are used to represent the vertices of valence 3 and 4 of a graph.

\begin{figure}[h]
\includegraphics{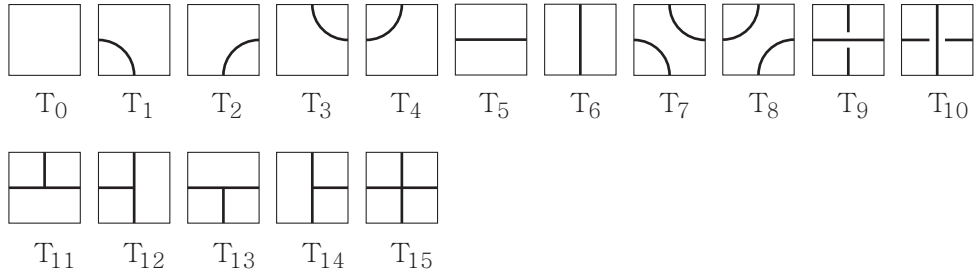}
\caption{Sixteen graph mosaic tiles}
\label{fig1}
\end{figure}

For positive integers $m$ and $n$,
an {\em $(m,n)$--mosaic\/} is an $m \times n$ matrix $M=(M_{ij})$ of graph mosaic tiles.
Note that the set of all $(m,n)$--mosaics has $16^{mn}$ elements.
A {\em connection point\/} of a graph mosaic tile is defined as the midpoint of a tile edge
that is also the endpoint of a portion of a graph drawn on the tile.
Note that graph mosaic tile $T_0$ has no connection point, $T_1$ through $T_6$ have two,
$T_{11}$ through $T_{14}$ have three, and $T_7$ through $T_{10}$ and $T_{15}$ have four.
A mosaic is called {\em suitably connected\/} if any pair of graph mosaic tiles
lying immediately next to each other in either the same row or the same column
have or do not have connection points simultaneously on their common edge.

A {\em graph $(m,n)$--mosaic\/} is a suitably connected $(m,n)$--mosaic
without connection points on the boundary.
Then this graph $(m,n)$--mosaic represents a specific graph diagram having vertices of valence 3 and 4.
Three figures in Figure \ref{fig2} represent a $(4,3)$--mosaic, a graph $(5,4)$--mosaic 
and the trefoil knot $(4,4)$--mosaic.

\begin{figure}[h]
\includegraphics{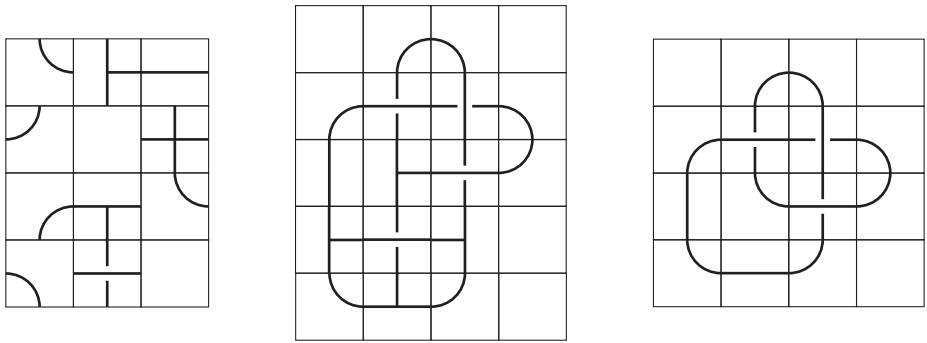}
\caption{Examples of mosaics}
\label{fig2}
\end{figure}

One of the natural problems in studying mosaic theory is the enumeration of certain type of mosaics.
$D^{(m,n)}_G$ denotes the total number of graph $(m,n)$--mosaics.
Remark that equivalent embedded graphs are counted as distinct 
if they are represented as mosaics in distinct ways. 
Originally, Lomonaco and Kauffman proposed the problem finding the number of all knot mosaics \cite{LK3}.
Indeed the total number of knot mosaics indicates the dimension of the Hilbert space of
so-called quantum knot mosaics.
In this paper, we present an algorithm producing the exact value of $D^{(m,n)}_G$.
Let $b(i,j)$ be the sum of the Hamming weights of the binary expansions of integers $i \! - \! 1$ and $j \! - \! 1$.
The Hamming weight of the binary number is the number of 1's in the string.
For example, $b(12,15) = 6$ because the binary expansions of 11 and 14 are 1011 and 1110.
Denote that $I_k$ is the identity matrix of size $2^k \times 2^k$
and $L_k=\frac{(1+\sqrt{5})^k+(1-\sqrt{5})^k}{2^k}$ is the Lucas number.

\begin{theorem} \label{thm:graph}
For positive integers $m$ and $n$, the total number $D^{(m+2,n+2)}_G$ of
all graph $(m \! + \! 2,n \! + \! 2)$--mosaics is  
$$D^{(m+2,n+2)}_G = \sum_{i, \, j=1}^{2^{m+n}} L_{b(i,j)} \, x_{ij},$$
where $x_{ij}$ is the $(i,j)$--entry of
the $2^{m+n} \times 2^{m+n}$ magnified state matrix $\widehat{N}^{(m,n)}_G$ recursively obtained by  
$$\widehat{N}^{(m,k+1)}_G =
\begin{bmatrix} \widehat{N}^{(m,k)}_G \cdot (I_k \otimes X_m^+) & 
\widehat{N}^{(m,k)}_G \cdot (I_k \otimes O_m^-) \\
\widehat{N}^{(m,k)}_G \cdot (I_k \otimes X_m^-) & 
\widehat{N}^{(m,k)}_G \cdot (I_k \otimes O_m^+) \end{bmatrix} $$
for $k=0,1,\dots, n-1$, with
$\widehat{N}_G^{(m,0)}= I_m$,
where the $2^m \times 2^m$ state matrices $X_m^+$, $X_m^-$, $O_m^+$ and $O_m^-$ are recursively obtained by  
$$ X_{k+1}^+ = \begin{bmatrix} X_k^+ & O_k^- \\ O_k^- & X_k^+ + O_k^- \end{bmatrix}, \ \  \ \ \
   X_{k+1}^- = \begin{bmatrix} X_k^- & O_k^+ \\ O_k^+ & X_k^- + O_k^+ \end{bmatrix}, $$
$$ O_{k+1}^+ = \begin{bmatrix} O_k^+ & X_k^- + O_k^+ \\
   X_k^- + O_k^+ &  X_k^- + 5 \, O_k^+ \end{bmatrix} \ \mbox{and} \ \
   O_{k+1}^- = \begin{bmatrix} O_k^- & X_k^+ + O_k^- \\
   X_k^+ + O_k^- & X_k^+ + 5 \, O_k^- \end{bmatrix} $$
for $k=0,1,\dots, m-1$, with
$X_0^+ = O_0^+ = \begin{bmatrix} 1 \end{bmatrix}$ and
$X_0^- = O_0^- = \begin{bmatrix} 0 \end{bmatrix}$.
\end{theorem}

Due to Theorem \ref{thm:graph}, we get Table \ref{tab1} of the first 8 terms
of the precise values of $D^{(n,n)}_G$.

\begin{table}[h]
\bgroup
\def\arraystretch{1.1} 
{\small 
\begin{tabular}{cl}      \hline \hline
\ \ \ $n$ \ \ \  & $D^{(n,n)}_G$  \\    \hline
1 & 1 \\
2 & 2 \\
3 & 71 \\
4 & 144212 \\
5 & 9899808106 \\
6 & 21965008855047380 \\
7 & 1573773836263642972028928 \\
8 & 3640808935014382048919715166814208 \ \ \ \\   \hline \hline
\end{tabular}
}
\egroup
\vspace{4mm}
\caption{List of $D^{(n,n)}_G$}
\label{tab1}
\end{table}

\section{State matrices $X_m^+$, $X_m^-$, $O_m^+$ and $O_m^-$}

The notation and terminology used in this paper will be the same as 
those employed in the previous paper~\cite{OHLL},
but differ partly to adjust into graph objects.
Let $m$ and $n$ be positive integers.
$\mathbb{S}^{(m,n)}_G$ denotes the set of all suitably connected $(m,n)$--mosaics,
which possibly have connection points on their boundary edges.
A suitably connected (5,3)--mosaic $S^{(5,3)}$ is drawn in Figure \ref{fig3}.

\begin{figure}[h]
\includegraphics{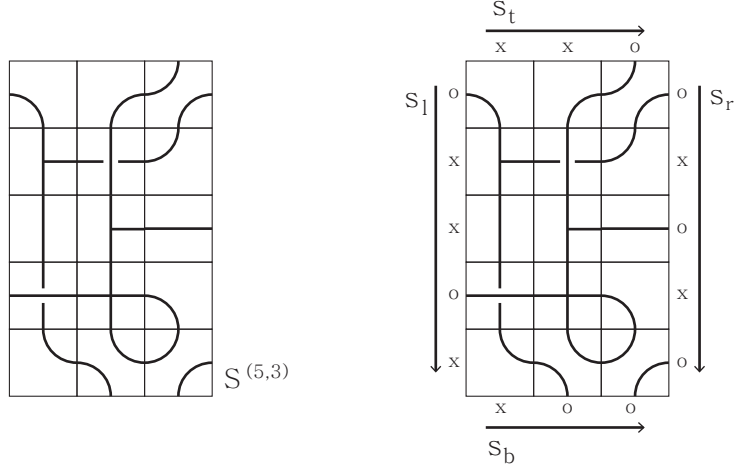}
\caption{Suitably connected (5,3)--mosaic $S^{(5,3)}$ with state indications}
\label{fig3}
\end{figure}

For simplicity of exposition, a graph mosaic tile is called $l$--, $r$--, $t$-- and $b$--{\em cp\/}, respectively,
if it has a connection point on its left, right, top and bottom boundaries.
We open use two or more letters such as $lt$--cp for the case of both $l$--cp and $t$--cp.
Also we use the sign $\hat{}$ \/ for negation so that, for example,
$\hat{l}$--cp means not $l$--cp, and
$\hat{l} \hat{t}$--cp means both $\hat{l}$--cp and $\hat{t}$--cp.
\vspace{3mm}

\noindent {\bf Choice rule for graph.\/}
Each $M_{ij}$ in a mosaic has five choices of the graph mosaic tiles
as $T_7$, $T_8$, $T_9$, $T_{10}$ and $T_{15}$ if it has four connection points (i.e. $lrtb$--cp).
It has a unique choice if it has no, two or three connection points.
It does not have exactly one connection point.
\vspace{3mm}

For a suitably connected $(m,n)$--mosaic $S^{(m,n)} =  (M_{ij})$ where $i=1,\dots,m$ and $j=1,\dots,n$,
an {\em $l$--state\/} of $S^{(m,n)}$ indicates the presence of connection points
on its boundary edge on the left,
and we denote that $s_l(S^{(m,n)}) = s_l(M_{11}) s_l(M_{21}) \cdots s_l(M_{m1})$
where $s_l(M_{ij})$ denotes ``x'' if $M_{ij}$ is $\hat{l}$--cp and ``o'' if $M_{ij}$ is $l$--cp.
Similarly we define $r$--, $t$-- and $b$--states of $S^{(m,n)}$, respectively,
that indicate the presences of connection points on its boundary edge
on the right, top and bottom.
As an example, the suitably connected $(5,3)$--mosaic $S^{(5,3)}$ drawn in Figure \ref{fig3} has
$s_l(S^{(5,3)}) =$ oxxox, $s_r(S^{(5,3)}) =$ oxoxo, $s_t(S^{(5,3)}) =$ xxo and $s_b(S^{(5,3)}) =$ xoo.
Obviously $\mathbb{S}^{(m,n)}_G$ has possibly $2^m$ kinds of $l$--states, $2^m$ kinds of $r$--states,
$2^n$ kinds of $t$--states and $2^n$ kinds of $b$--states.
Arrange the elements of the set of states in the reverse lexicographical order,
for example, xxx, oxx, xox, oox, xxo, oxo, xoo and ooo when $m=3$.

Now we define the {\em state matrices\/} $X_m^+$, $X_m^-$, $O_m^+$ and $O_m^-$ for $\mathbb{S}^{(m,1)}_G$.
$X_m^+$ ($X_m^-$, $O_m^+$ or $O_m^-$) is a $2^m \times 2^m$ matrix $(x_{ij})$
where each entry $x_{ij}$ indicates the number of all suitably connected $(m,1)$--mosaics
that have the $i$-th $l$--state and the $j$-th $r$--state in the set of $2^m$ states
of the order arranged above, 
and additionally whose ($b$--state, $t$--state) is (x, x) ((x, o), (o, o) or (o, x), respectively).
The letters $X$ and $O$ show their $b$--states and
we use the sign $\mbox{}^+$ (or $\mbox{}^-$)
when they have the same (different, respectively) $b$--state and $t$--state.
For a $2^{k+1} \times 2^{k+1}$ matrix $N = (x_{ij})$,
the 11--quadrant (similarly 12--, 21-- or 22--quadrant) of $N$ denotes
the $2^k \times 2^k$ quarter matrix $(x_{ij})$ where $1 \leq i, j \leq 2^k$
($1 \leq i \leq 2^k$ and $2^k+1 \leq j \leq 2^{k+1}$,
$2^k+1 \leq i \leq 2^{k+1}$ and $1 \leq j \leq 2^k$,
or $2^k+1 \leq i, j \leq 2^{k+1}$, respectively).

\begin{proposition} \label{prop:m1}
For the set $\mathbb{S}^{(m,1)}_G$ of all suitably connected $(m,1)$--mosaics,
the associated state matrices $X_m^+$, $X_m^-$, $O_m^+$ and $O_m^-$ are obtained by
$$ X_{k+1}^+ = \begin{bmatrix} X_k^+ & O_k^- \\ O_k^- & X_k^+ + O_k^- \end{bmatrix}, \ \  \ \ \
   X_{k+1}^- = \begin{bmatrix} X_k^- & O_k^+ \\ O_k^+ & X_k^- + O_k^+ \end{bmatrix}, $$
$$ O_{k+1}^+ = \begin{bmatrix} O_k^+ & X_k^- + O_k^+ \\
   X_k^- + O_k^+ &  X_k^- + 5 \, O_k^+ \end{bmatrix} \ \mbox{and} \ \
   O_{k+1}^- = \begin{bmatrix} O_k^- & X_k^+ + O_k^- \\
   X_k^+ + O_k^- & X_k^+ + 5 \, O_k^- \end{bmatrix} $$
for $k=1,2,\dots, m-1$, starting with
$$X_1^+ = \begin{bmatrix} 1 & 0 \\ 0 & 1 \end{bmatrix}, \
X_1^- = \begin{bmatrix} 0 & 1 \\ 1 & 1 \end{bmatrix}, \
O_1^+ = \begin{bmatrix} 1 & 1 \\ 1 & 5 \end{bmatrix} \mbox{and} \ \
O_1^- = \begin{bmatrix} 0 & 1 \\ 1 & 1 \end{bmatrix}.$$
\end{proposition}

Note that for the initial condition of this recursion formula,
we may start at $k=0$ with $X_0^+ = O_0^+ = \begin{bmatrix} 1 \end{bmatrix}$ and
$X_0^- = O_0^- = \begin{bmatrix} 0 \end{bmatrix}$.

\begin{proof}
We follow the proof of Proposition 2 in \cite{OHLL} with some extension.
Use induction on $m$.
First, we establish the four state matrices $X_1^+$, $X_1^-$, $O_1^+$ and $O_1^-$ 
for $(1,1)$--mosaics directly from Figure \ref{fig4}.
The entries 0, 1 and 5 are obtained by Choice rule for graph.
For example, $(2,2)$--entry of $O_1^+$ is 5 because the associated graph mosaic tiles must be $lrtb$--cp.
Note that the sum of all entries in these four matrices is
the total number of elements of $\mathbb{S}^{(1,1)}_G$ that is obviously 16.
Indeed, each element of $\mathbb{S}^{(1,1)}_G$ can be extended to
graph $(3,3)$--mosaics in several ways depending on the number of connection points
on its boundary edges.
This will be considered in Section 4.

\begin{figure}[h]
\includegraphics{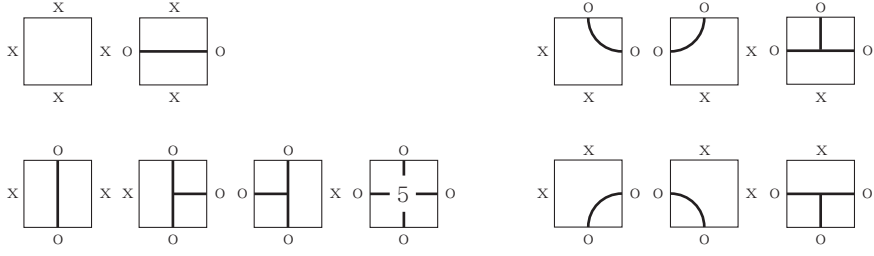}
\caption{$(1,1)$--mosaics for $X_1^+$(top-left), 
$X_1^-$(top-right), $O_1^+$(bottom-left) and $O_1^-$(bottom-right)}
\label{fig4}
\end{figure}

Assume that the state matrices $X_k^+$, $X_k^-$, $O_k^+$ and $O_k^-$ satisfy the statement.
We find $O_{k+1}^+$, and the readers can easily apply this method
to find $X_{k+1}^+$, $X_{k+1}^-$ and $O_{k+1}^-$.
All entries of $O_{k+1}^+$ count the suitably connected $(k+1,1)$--mosaics in $\mathbb{S}^{(k+1,1)}_G$ 
whose ($b$--state, $t$--state) is (o, o).
Let $S^{(k+1,1)} = (M_{i,1})$ be a suitably connected $(k \! + \! 1,1)$--mosaic.
Focus on the bottom graph mosaic tile $M_{k+1,1}$.
If it is $\hat{l} \hat{r}$--cp,
then $S^{(k+1,1)}$ should be counted in an entry of the 11--quadrant of $O_{k+1}^+$.
This is because of the reverse lexicographical order of $2^{k+1}$ states.
In this case, $M_{k+1,1}$ must be the graph mosaic tile $T_6$.
Let $S^{(k,1)}$ be the associated suitably connected $(k,1)$--mosaic
obtained from $S^{(k+1,1)}$ by deleting $M_{k+1,1}$.
Since ($b$--state, $t$--state) of $S^{(k,1)}$ is  (o, o),
the associated state matrix for all possible such $S^{(k,1)}$ is $O_{k}^+$.

If $M_{k+1,1}$ is $lr$--cp,
then $S^{(k+1,1)}$ should be counted in an entry of the 22--quadrant of $O_{k+1}^+$.
In this case $M_{k+1,1}$ must be one of 
the graph mosaic tiles $T_7$, $T_8$, $T_9$, $T_{10}$, $T_{13}$ or $T_{15}$.
In this case, ($b$--state, $t$--state) of $S^{(k,1)}$ is
either (o, x) when $M_{k+1,1}$ is $T_{13}$ or (o, o) for the other five graph mosaic tiles.
Thus the associated state matrix for all possible such $S^{(k,1)}$ is $X_k^- + 5 \, O_k^+$.
For the case that $M_{k+1,1}$ is either $\hat{l} r$--cp or $l \hat{r}$--cp, we follow similar arguments.
Note that Figure \ref{fig5} and Table \ref{tab2} below explain all four cases according to
the $l$--state and the $r$--state of $M_{k+1,1}$.
\end{proof}

\begin{figure}[h]
\includegraphics{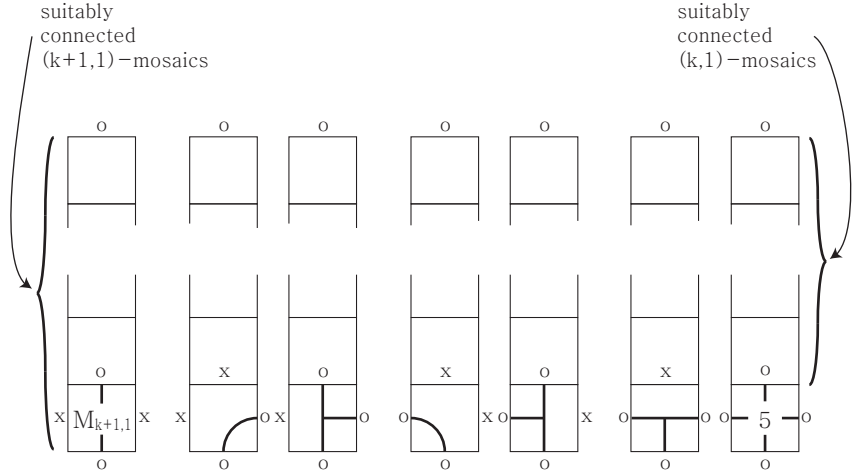}
\caption{Finding the matrix $O_{k+1}^+$}
\label{fig5}
\end{figure} 

\begin{table}[h]
\bgroup
\def\arraystretch{1.1} 
{\small 
\begin{tabular}{cccc}      \hline \hline
 & {\em Quadrants\/} & {\em Associated $M_{k+1,1}$\/} & \ {\em Submatrices\/} \ \\    \hline
\multirow{4}{7mm}{$O_{k+1}^+$}
 &  \ 11--quadrant ($\hat{l} \hat{r} b$--cp) \ & $T_6$ & $O_k^+$ \\
 & 12--quadrant ($\hat{l} r b$--cp) & $T_2$ or $T_{14}$ & $X_k^- + O_k^+$ \\
 & 21--quadrant ($l \hat{r} b$--cp) & $T_1$ or $T_{12}$ & $X_k^- + O_k^+$ \\
 & 22--quadrant ($l r b$--cp) & \ $T_{13}$ or $T_7 \sim T_{10}$, $T_{15}$ \ & $X_k^- + 5  O_k^+$ \\  \hline \hline
\end{tabular}
}
\egroup
\vspace{4mm}
\caption{Four quadrants of $O_{k+1}^+$}
\label{tab2}
\end{table}

\section{Magnified state matrix $\widehat{N}^{(m,n)}_G$}

For a suitably connected $(m,n)$--mosaic $S^{(m,n)} =  (M_{ij})$,
we define an {\em $lt$--state\/} of $S^{(m,n)}$ by $s_{lt}(S^{(m,n)}) = s_l(S^{(m,n)}) s_t(S^{(m,n)})$,
and similarly define an {\em $rb$--state\/} of $S^{(m,n)}$.
For example, $s_{lt}(S^{(5,3)}) =$ oxxoxxxo and $s_{rb}(S^{(5,3)}) =$ oxoxoxoo
for $S^{(5,3)}$ drawn in Figure \ref{fig3}.
Note that $\mathbb{S}^{(m,n)}_G$ has possibly $2^{m+n}$ kinds of $lt$--states
and $2^{m+n}$ kinds of $rb$--states.
We arrange the elements of the set of all states in the reverse lexicographical order
as before.

A {\em magnified state matrix\/} $\widehat{N}^{(m,n)}_G$ for $\mathbb{S}^{(m,n)}_G$  is
a $2^{m+n} \times 2^{m+n}$ matrix $(x_{ij})$
where each entry $x_{ij}$ is the number of all suitably connected $(m,n)$--mosaics
that have the $i$-th $lt$--state and the $j$-th $rb$--state in the set of $2^{m+n}$ states.

Like a checkerboard,
we divide $\widehat{N}^{(m,n)}_G$ into $2^{2n}$ submatrices $N^{(m,n)}_{uv}$, $u,v=1, \dots, 2^n$,
each of which indicates a $2^m \times 2^m$ state matrix for the set of
suitably connected $(m,n)$--mosaics that have the $u$-th $t$--state and the $v$-th $b$--state
in the set of $2^n$ states.
We write $\widehat{N}^{(m,n)}_G = [N^{(m,n)}_{uv}]$ as a $2^n \times 2^n$ matrix 
whose entries are also $2^m \times 2^m$ matrices.
All submatrices are related to all possible $2^{2n}$ choices of $t$--states and $b$--states.

\begin{proposition} \label{prop:mn}
For the set $\mathbb{S}^{(m,n)}_G$ of all suitably connected $(m,n)$--mosaics,
the associated magnified state matrix $\widehat{N}^{(m,n)}_G$ is obtained by
$$\widehat{N}^{(m,k+1)}_G =
\begin{bmatrix} \widehat{N}^{(m,k)}_G \cdot (I_k \otimes X_m^+) & 
\widehat{N}^{(m,k)}_G \cdot (I_k \otimes O_m^-) \\
\widehat{N}^{(m,k)}_G \cdot (I_k \otimes X_m^-) & 
\widehat{N}^{(m,k)}_G \cdot (I_k \otimes O_m^+) \end{bmatrix} $$
for $k=1,2,\dots, n \! - \! 1$, with
$\widehat{N}_G^{(m,1)}= \begin{bmatrix} X_m^+ & O_m^- \\ X_m^- & O_m^+ \end{bmatrix}$.
\end{proposition}

Note that for the initial condition of this recursion formula,
we may start at $k=0$ with $\widehat{N}_G^{(m,0)}= \begin{bmatrix} I_m \end{bmatrix}$.

\begin{proof}
We use induction on $n$.
First, we establish the magnified state matrix $\widehat{N}^{(m,1)}_G$ for $\mathbb{S}^{(m,1)}_G$.
A suitably connected $(m,1)$--mosaic has two choices of $t$--state and $b$--state each among x or o.
These letters are related to the last letters of the words of the $lt$--state and the $rb$--state of the mosaic.
Therefore the 11--quadrant of $\widehat{N}^{(m,1)}_G$ counts only suitably connected $(m,1)$--mosaic
whose ($t$--state, $b$--state) is (x,x).
This means that it is $X_m^+$.
Similarly the 12--, 21-- and 22--quadrants of $\widehat{N}^{(m,1)}_G$ are $O_m^-$, $X_m^-$ and
$O_m^+$, respectively.

Assume that $\widehat{N}^{(m,k)}_G = [N^{(m,k)}_{uv}]$, $u,v=1, \dots, 2^k$, satisfies the statement.
We follow the proof of Proposition 3 in \cite{OHLL} with some extension.
Let $S^{(m,k+1)}$ be a suitably connected $(m,k \! + \! 1)$--mosaic in $\mathbb{S}^{(m,k+1)}_G$.
Also let $S^{(m,k)}$ be the suitably connected $(m,k)$--mosaic obtained by
deleting the rightmost column of $S^{(m,k+1)}$, and $S^{(m,1)}$ be the suitably connected $(m,1)$--mosaic
obtained by taking only the rightmost column of $S^{(m,k+1)}$.
Then the $r$--state of $S^{(m,k)}$ must be the same as the $l$--state of $S^{(m,1)}$ as shown in Figure \ref{fig6}.

\begin{figure}[h]
\includegraphics{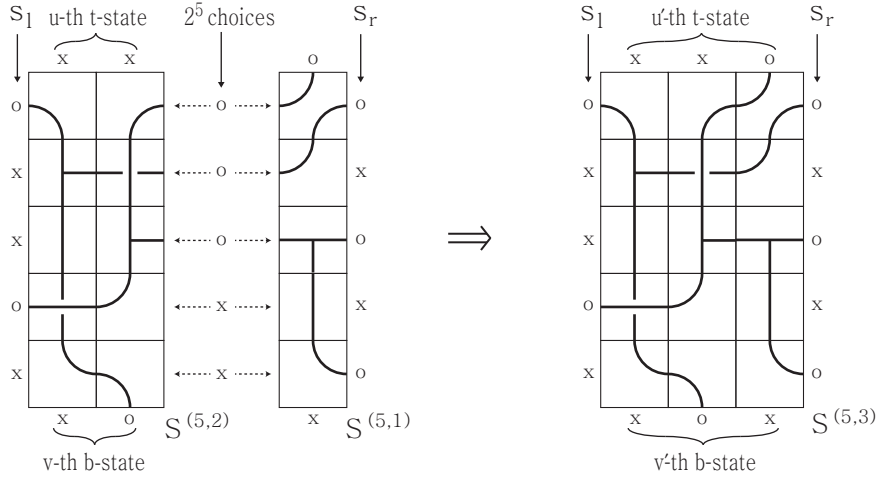}
\caption{Adjoining two suitably connected mosaics}
\label{fig6}
\end{figure}

Let $\widehat{N}^{(m,k+1)}_G = [N^{(m,k+1)}_{u'v'}]$ for $u',v'=1, \dots, 2^{k+1}$.
Now consider the 21--quadrant of $\widehat{N}^{(m,k+1)}_G$, for example.
Its entries are $2^{2k}$ submatrices $N^{(m,k+1)}_{u'v'}$
for $u'=2^k \! + \! 1, \dots, 2^{k+1}$ and $v'=1, \dots, 2^k$.
Associated suitably connected $(m,k+1)$--mosaics $S^{(m,k+1)}$ have the following property that
the last letters of their $lt$--states and $rb$--states are o and x, respectively.
This means that the $t$--state and the $b$--state of the rightmost column $S^{(m,1)}$ are o and x,
respectively as drawn in the figure.
Furthermore the associated $S^{(m,k)}$ must have the $u$-th $t$--state and the $v$-th $b$--state
where $u' = u + 2^k$ and $v' = v$ for this example.
This is because the $u$-th state and the $v$-th state are obtained from the words of
the $u'$-th state and the $v'$-th state by ignoring their last letters o and x, respectively.

Let $N^{(m,k+1)}_{u'v'} = (z_{ij})$, $N^{(m,k)}_{uv} = (y_{ij})$
and $X_m^- = (x_{ij})$ for $i,j = 1, \dots, 2^m$.
Among suitably connected $(m,k+1)$--mosaics counted in each entry $z_{ij}$,
the number of such mosaics whose $r$--state of the $k$-th column
(or equally $l$--state of the rightmost column) is the $q$-th state in the set of $2^m$ states
is the product of $y_{iq}$ and $x_{qj}$.
Since all $2^m$ states can be appeared as states of connection points
where $S^{(m,k)}$ and $S^{(m,1)}$ meet,
we get $z_{ij} = \sum^{2^m}_{q=1} y_{iq} x_{qj}$.
This implies that $N^{(m,k+1)}_{u'v'} = N^{(m,k)}_{uv} \, X_m^-$.
Therefore
$$ \mbox{21--quadrant of} \ \widehat{N}^{(m,k+1)}_G = 
\begin{bmatrix} N^{(m,k)}_{uv} \, X_m^- \end{bmatrix} =
\widehat{N}^{(m,k)}_G \cdot (I_k \otimes X_m^-). $$
Similarly we multiply $X_m^+$, $O_m^-$ and $O_m^+$ instead of $X_m^-$
for the 11--, 12-- and 22--quadrants of $\widehat{N}^{(m,k+1)}_G$, respectively.
We finish the proof.
\end{proof}

\section{Proof of Theorem \ref{thm:graph}}

Each suitably connected $(m,n)$--mosaic of $\mathbb{S}^{(m,n)}_G$ can be extended to
several graph $(m \! + \! 2,n \! + \! 2)$--mosaics by adjoining proper graph mosaic tiles surrounding it,
called boundary graph mosaic tiles, as illustrated in Figure \ref{fig7}.
Indeed the number of all possible extended graph mosaics
only depends on the number of connection points on the boundary edges.
For example, as drawn in the figure, there are exactly three ways to extend to graph mosaics
if a suitably connected mosaic has two connection points on the boundary.
This is a different aspect to knot mosaics,
for which there are always two ways to extend regardless of the number of connection points 
on the boundary, known as the two fold rule.

\begin{figure}[h]
\includegraphics{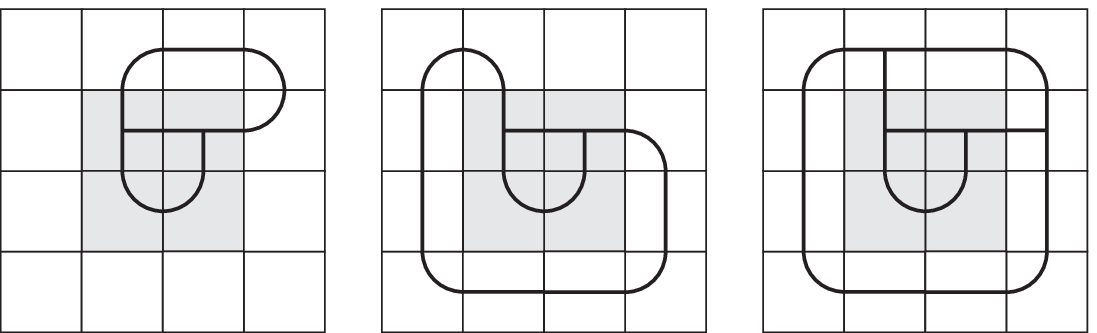}
\caption{Adjoining boundary graph mosaic tiles}
\label{fig7}
\end{figure}

Let $F_k$ be the number of choices of making graph $(m \! + \! 2,n \! + \! 2)$--mosaics
from a suitably connected $(m,n)$--mosaic that has $k$ connection points on the boundary.

\begin{proposition} \label{prop:pk}
$F_k$ is obtained by $$F_t=F_{t-1}+F_{t-2}$$
for $t=2,3,\dots, k$, starting with $F_0=2$, $F_1=1$.
\end{proposition}

\begin{proof}
Let $S^{(m,n)}$ be a suitably connected $(m,n)$--mosaic
that has $t$ connection points on the boundary.
We name these connection points by $1, \dots, t$ clockwise on the boundary.
To get a new graph $(m \! + \! 2,n \! + \! 2)$--mosaic,
we adjoin $2m+2n+4$ boundary graph mosaic tiles surrounding $S^{(m,n)}$
so that the new one has to be also suitably connected and
does not have connection points on the new boundary.
As drawn in Figure \ref{fig8},
each pair of consecutive connection points, say $\{i, i \!+\! 1\}$
(we use 1 for $i \! + \! 1$ if $i=t$), is
either joined by an arc on the boundary tiles between them or not.
The former is called {\em bridged\/} and the latter {\em unbridged\/}.
The right figure illustrates simplified graphical symbols of arcs
on a portion of boundary graph mosaic tiles.
Note that any successive pairs $\{i \!-\! 1,i\}$ and $\{i,i \!+\! 1\}$ can not be unbridged simultaneously
because of suitable connectedness.

\begin{figure}[h]
\includegraphics{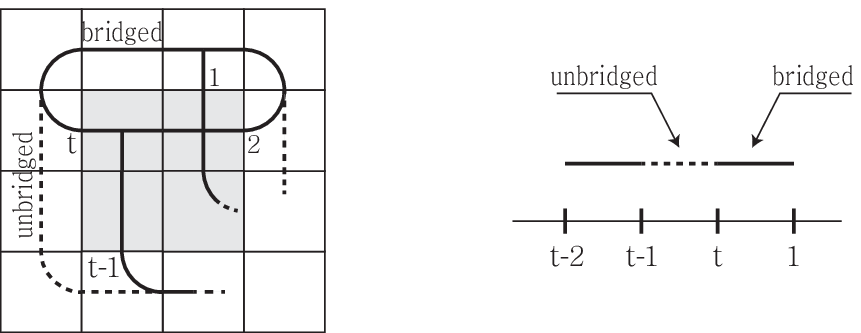}
\caption{Adjoining boundary graph mosaic tiles}
\label{fig8}
\end{figure}

$F_t$ is indeed the number of all possible choices of determining bridged or unbridged 
on $t$ pairs of consecutive connection points
with a restriction that any successive pairs can not be unbridged simultaneously.
The readers find that $F_0=2$, $F_1=1$, $F_2=3$ and $F_3=4$ by simple drawings
(Figure \ref{fig7} explains $F_2=3$).
Assume that $t \geq 4$ so that 1 and $t \! - \! 2$ are different.
Now we give a pictorial proof as illustrated in Figure \ref{fig9}.
We consider three situations of $S^{(m,n)}$ each of which has $t$, $t \! - \! 1$ or $t \! - \! 2$ connection points
on its boundary edges for counting $F_t$, $F_{t-1}$ and $F_{t-2}$, respectively.
The left five stacks in the figure indicate all five possibilities of bridged or unbridged
on three pairs $\{t \!-\! 2,t \!-\! 1\}$,  $\{t \!-\! 1,t\}$ and $\{t,1\}$ of consecutive connection points between $t-2$ and 1.
The middle three stacks and the right two stacks indicate all three and two possibilities of bridged or unbridged
on two pairs $\{t \!-\! 2,t \!-\! 1\}$ and $\{t \!-\! 1,1\}$, and a pair $\{t \!-\! 2,1\}$, respectively.

\begin{figure}[h]
\includegraphics{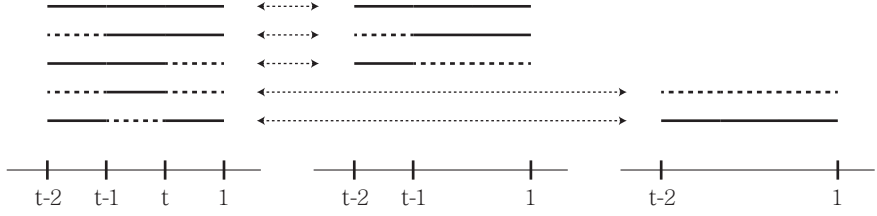}
\caption{$F_t=F_{t-1}+F_{t-2}$}
\label{fig9}
\end{figure}

Focus on the line segments indicating bridged or unbridged on the upper three stacks
of the left and the middle figures.
We notice that, in each line segment,
both pairs $\{t \!-\! 2,t \!-\! 1\}$'s for $F_t$ and $F_{t-1}$ are simultaneously either bridged  or unbridged,
and also both pairs $\{t ,1\}$ for $F_t$ and $\{t \!-\! 1,1\}$ for $F_{t-1}$ are 
simultaneously either bridged  or unbridged.
In each case, $F_t$ and $F_{t-1}$ involve the same number of choices of determining bridged or unbridged 
on their rest $t \! - \! 3$ pairs $\{1,2\}$ through $\{t \!-\! 3,t \!-\! 2\}$ extended properly.
This implies that the total number of choices associated with these upper three cases among $F_t$ choices 
is exactly the same as $F_{t-1}$.
Now consider the line segments on the bottom two stacks of the left and the right figures.
In both figures, two pairs $\{t \!-\! 2,t \!-\! 1\}$ and $\{t,1\}$ for $F_t$ and a pair $\{t \!-\! 2,1\}$ for $F_{t-2}$
are simultaneously either bridged or unbridged.
In each case we have the same number of choices determining bridged or unbridged on the rest pairs as before.
This implies that the total number of choices associated with these two cases among $F_t$ choices 
is the same as $F_{t-2}$.
Thus we obtain $F_t=F_{t-1}+F_{t-2}$.
\end{proof}

Note that $F_k$ generates the Lucas number $L_k=\frac{(1+\sqrt{5})^k+(1-\sqrt{5})^k}{2^k}$.

\begin{proof}[Proof of Theorem \ref{thm:graph}]
Let $m$ and $n$ be any positive integers.
Consider the set $\mathbb{S}^{(m,n)}_G$ of all suitably connected $(m,n)$--mosaics
and the associated magnified state matrix $\widehat{N}^{(m,n)}_G = (x_{ij})$, $i,j= 1, \dots, 2^{m+n}$,
obtained from Propositions~\ref{prop:m1} and \ref{prop:mn}.
By the definition of the magnified state matrix,
each entry $x_{ij}$ counts the number of all suitably connected $(m,n)$--mosaics
that have the $i$-th $lt$--state and the $j$-th $rb$--state in the set of all $2^{m+n}$ states.
Note that the total number of elements of $\mathbb{S}^{(m,n)}$ is
the sum of all entries of $\widehat{N}^{(m,n)}_G$
because all rows represent all $2^{m+n}$ $lt$--states
and all columns represent all $2^{m+n}$ $rb$--states of mosaics of $\mathbb{S}^{(m,n)}_G$.

From the definition of the order of states, the $i$-th state in the set of $2^{m+n}$ states is obtained as follows:
(1) take the binary numeral system of $i \! - \! 1$,
(2) replace the numbers 0 and 1 by the letters x and o, respectively,
(3) write the word in reverse,
and then (4) adjoin the letter x repeatedly at the end of the word
until the total number of the letters of the word is $m \! + \! n$.
Let $b(i,j)$ be the total number of letter 1's appeared in the binary numeral systems of $i \! - \! 1$ and $j \! - \! 1$.
This guarantees that a suitably connected mosaic
that has the $i$-th $lt$--state and the $j$-th $rb$--state has $b(i,j)$ connection points on the boundary.

Therefore each suitably connected mosaic of $\mathbb{S}^{(m,n)}$
that has the $i$-th $lt$--state and the $j$-th $rb$--state
can be extended to exactly $F_{b(i,j)}$ graph $(m \! + \! 2,n \! + \! 2)$--mosaics.
By Proposition \ref{prop:pk}, the total number $D^{(m+2,n+2)}_G$ of 
all graph $(m \! + \! 2,n \! + \! 2)$--mosaics is obtained by
$$D^{(m+2,n+2)}_G = \sum_{i, \, j=1}^{2^{m+n}} L_{b(i,j)} \, x_{ij}. $$
This completes the proof.
\end{proof}

\end{document}